\documentclass[12pt,reqno]{amsart} 
\usepackage{amssymb}
\usepackage{mathrsfs}
\usepackage{enumerate}
\usepackage[all]{xy}
\usepackage[usenames,dvipsnames]{color}
\usepackage[colorlinks=true, citecolor=OliveGreen, 
linkcolor=OliveGreen]{hyperref}

\renewcommand{\today}{\the\day/\the\month/\the\year}

\makeatletter
\@namedef{subjclassname@2020}{\textup{2020} Mathematics Subject 
Classification} 
\makeatother

\DeclareMathAlphabet\EuR{U}{eur}{m}{n}
\SetMathAlphabet\EuR{bold}{U}{eur}{b}{n}
\newcommand{\curs}{\EuR}
\newcommand{\catdef}[2][]{\expandafter\newcommand\csname#2\endcsname%
{#1\curs{#2}}}
\catdef{Ab}    \catdef{Grp}   \catdef{Cat}   \catdef{Vect}   \catdef{Mod}
\catdef{Is}    \catdef{Top}   \catdef{hoTop}  \catdef{ho}   \catdef{FibSimp}

\let\Gamma=\varGamma
\let\Omega=\varOmega
\let\Sigma=\varSigma

\renewenvironment{enumerate}[1][]
{\begin{enumerat}[#1]\setlength{\itemsep}{6pt}}{\end{enumerat}}

\newenvironment{enuma}{\begin{enumerate}[{\rm(a) }]}{\end{enumerate}}

\renewenvironment{itemize}
{\begin{itemiz}\setlength{\itemsep}{6pt} \setlength{\itemindent}{-5pt} }
{\end{itemiz}}

\definecolor{darkgreen}{rgb}{0,0.5,0}
\definecolor{bluegreen}{rgb}{0,0.2,0.8}
\definecolor{darkred}{rgb}{0.8,0,0}
\definecolor{newercolor}{rgb}{0.2,0,1}
\definecolor{darkyellow}{rgb}{0.7,0.7,0}
\definecolor{darkorange}{rgb}{0.8,0.4,0}

\newcommand{\8}[1]{c_{#1}}

\numberwithin{table}{section}

\newlength{\short}
\newcommand{\4}[1]{\widebar{#1}}
\newcommand{\5}[1]{\widehat{#1}}

\def\pair[#1,#2]{[\hskip-1.5pt[#1,#2]\hskip-1.5pt]}

\SelectTips{cm}{10} \UseTips   

\let\oldcirc=\circ
\renewcommand{\circ}{\mathchoice
    {\mathbin{\scriptstyle\oldcirc}}{\mathbin{\scriptstyle\oldcirc}}
    {\mathbin{\scriptscriptstyle\oldcirc}}
    {\mathbin{\scriptscriptstyle\oldcirc}}}

\def\beq#1\eeq{\begin{equation*}#1\end{equation*}}
\def\beqq#1\eeqq{\begin{equation}#1\end{equation}}

\numberwithin{equation}{section}

\newtheorem{Thm}{Theorem}[section]
\newtheorem{Prop}[Thm]{Proposition}

\newtheorem{Lem}[Thm]{Lemma}

\theoremstyle{definition}
\newtheorem{Defi}[Thm]{Definition}

\newcommand{\widebar}[1]
      {\overset{{\mskip1mu\leaders\hrule height0.4pt\hfill\mskip1mu}}{#1}
      \vphantom{#1}}

\newcounter{let} 
\loop\stepcounter{let}
\expandafter\edef\csname cal\alph{let}\endcsname%
{\noexpand\mathcal{\Alph{let}}}
\ifnum\thelet<26\repeat

\loop\stepcounter{let}
\expandafter\edef\csname scr\alph{let}\endcsname%
{\noexpand\mathscr{\Alph{let}}}
\ifnum\thelet<26\repeat

\newcommand{\tdef}[2][]{\expandafter\newcommand\csname#2\endcsname%
{#1\textup{#2}}}
\tdef{Iso}   \tdef{Aut}    \tdef{Out}    \tdef{Inn}    \tdef{Hom}
\tdef{End}   \tdef{Inj}    \tdef{map}    \tdef{Ker}    \tdef{Ob}
\tdef{Mor}   \tdef{Res}    \tdef{Id}     \tdef{Fr}     \tdef{Spin} 
\tdef{rk}    \tdef{conj}   \tdef{incl}   \tdef{proj}   \tdef{diag} 
\tdef{trf}   \tdef{Sol}    \tdef{He}     \tdef{Sz}     \tdef{cj}
\tdef{Rep}   \tdef{pr}    \tdef{Inndiag} \tdef{Outdiag}  \tdef{expt}
\tdef{supp}  \tdef{Isom}   \tdef{ord}    \tdef{Coker}   \tdef{Tr}
\tdef[_]{typ} \tdef[^]{op} \tdef[^]{ab}   \tdef{lcm}  \tdef{McL}
\tdef{restr}  \tdef{Comp}  \tdef{HS}     \tdef{ev}    \tdef{Stab}
\tdef{srk}

\newcommand{\fdef}[1]{\expandafter\newcommand\csname#1\endcsname%
{\mathfrak{#1}}}
\fdef{X}  \fdef{red}  \fdef{foc}  \fdef{hyp}  \fdef{Lie} \fdef{Y}

\newcommand{\bbdef}[1]{\expandafter\newcommand%
\csname#1\endcsname{\mathbb{#1}}}
\bbdef{C} \bbdef{F} \bbdef{R} \bbdef{Z} \bbdef{N} \bbdef{Q} \bbdef{K}

\newcommand{\itdef}[1]{\expandafter\newcommand\csname#1\endcsname%
{\textit{#1}}}
\itdef{PSL}  \itdef{PSU}  \itdef{SL}  \itdef{SU}  \itdef{GL} \itdef{GU}
\itdef{Sp}   \itdef{PSp} \itdef{PSO} \itdef{SO}   \itdef{SD} \itdef{PGU} 
\itdef{PGL}  \itdef{Co}  \itdef{Fi}  \itdef{GO}   \itdef{BDI}

\newcommand{\gee}{\varepsilon}
\newcommand{\sminus}{\smallsetminus}
\newcommand{\lie}[3]{\def\test{#2}\def\tst{G}\ifx\test\tst{{}^{#1}#2_{#3}}
\else{{}^{#1}\!#2_{#3}}\fi}
\renewcommand{\*}{\,\lower6pt\hbox{\Large{\textup{*}}}\,}
\newcommand{\syl}[3][]{\textup{Syl}^{#1}_{#2}(#3)}
\newcommand{\sylp}[2][]{\syl[#1]{p}{#2}}

\renewcommand{\Im}{\textup{Im}}
\newcommand{\autf}{\Aut_{\calf}}

\newcommand{\homf}{\Hom_{\calf}}

\newcommand{\mxfoura}[8]{\left(\begin{smallmatrix}#1&#2&#3&#4\\#5&#6&#7&#8}
\newcommand{\mxfourb}[8]{\\#1&#2&#3&#4\\#5&#6&#7&#8\end{smallmatrix}\right)}

\renewcommand{\:}{\colon}

\newcommand{\nsg}{\trianglelefteq}

\let\too=\longrightarrow
\let\xto=\xrightarrow

\newcommand{\gen}[1]{{\langle}#1{\rangle}}

\newcommand{\longleft}[1]{\;{\leftarrow%
\count255=0 \loop \mathrel{\mkern-6mu}%
    \relbar\advance\count255 by1\ifnum\count255<#1\repeat}\;}
\newcommand{\longright}[1]{\;{\count255=0 \loop \relbar\mathrel{\mkern-6mu}%
    \advance\count255 by1\ifnum\count255<#1\repeat\rightarrow}\;}
\newcommand{\Right}[2]{\overset{#2}{\longright#1}}
\newcommand{\RIGHT}[3]{\mathrel{\mathop{\kern0pt\longright#1}
        \limits^{#2}_{#3}}}

\newcommand{\LEFT}[3]{\mathrel{\mathop{\kern0pt\longleft#1}\limits^{#2}_{#3}}
}
\newcommand{\dRIGHT}[3]{\mathrel{%
   \mathop{\vcenter{\baselineskip=0pt\hbox{$\kern0pt\longright#1$}%
   \hbox{$\kern0pt\longright#1$}}}\limits^{#2}_{#3}}}
\newcommand{\LRIGHT}[3]{\mathrel{%
   \mathop{\vcenter{\baselineskip=0pt\hbox{$\kern0pt\longleft#1$}%
   \hbox{$\kern0pt\longright#1$}}}\limits^{#2}_{#3}}}
\newcommand{\RLEFT}[3]{\mathrel{%
   \mathop{\vcenter{\baselineskip=0pt\hbox{$\kern0pt\longright#1$}%
   \hbox{$\kern0pt\longleft#1$}}}\limits^{#2}_{#3}}}
\newcommand{\onto}[1]{\;{\count255=0 \loop \relbar\mathrel{\mkern-6mu}%
    \advance\count255 by1
    \ifnum\count255<#1 \repeat \twoheadrightarrow}\;}

\newcommand{\longline}{\bigskip\hfill\hbox to 8cm{\hrulefill}%
\hfill\bigskip}

\def\LFS(#1){\textup{LFS($#1$)}} 
\def\LF(#1){\textup{LF($#1$)}} 

\fdef{Fin}
\fdef{I}
\tdef{Ext}

\newcommand{\higherlim}[2]{\displaystyle\setbox1=\hbox{\rm lim}
	\setbox2=\hbox to \wd1{\leftarrowfill} \ht2=0pt \dp2=-1pt
	\setbox3=\hbox{$\scriptstyle{#1}$}
	\ifdim\wd1<\wd3
	\mathop{\hphantom{^{#2}}\vtop{\baselineskip=5pt\box1\box2}^{#2}}_{#1}
	\else
	\mathop{\vtop{\baselineskip=5pt\box1\box2}^{#2}}_{#1}
	\fi}

\begin{document}

\title{Correction to: An algebraic model for finite loop spaces}

\author{Carles Broto}
\address{Departament de Matem\`atiques, Universitat Aut\`onoma de Barcelona,
Edifici C, 08193 Bellaterra, Spain and Centre de Recerca Matem\`atica,
Edifici C, Campus Bellaterra, 08193 Bellaterra, Spain}

\email{carles.broto@uab.cat}
\thanks{C. Broto is partially supported by the Spanish State Research Agency
through the Severo Ochoa and Mar\'ia de Maeztu Program for Centers and Units
of Excellence in R\&D (CEX2020-001084-M), by AEI grant PID2020-116481GB-100,
and by AGAUR grant 2021-SGR-01015.}

\author{Ran Levi}
\address{Institute of Mathematics, University of Aberdeen,
Fraser Noble 138, Aberdeen AB24 3UE, U.K.}
\email{r.levi@abdn.ac.uk}
\thanks{}

\author{Bob Oliver}
\address{Universit\'e Sorbonne Paris Nord, LAGA, UMR 7539 du CNRS, 
99, Av. J.-B. Cl\'ement, 93430 Villetaneuse, France.}
\email{bobol@math.univ-paris13.fr}
\thanks{B. Oliver is partially supported by UMR 7539 of the CNRS}





\begin{abstract}
We correct here two errors in our earlier paper "An algebraic model for 
finite loop spaces". 
\end{abstract}


\bigskip

\maketitle

In this paper, we correct two erroneous arguments found in our earlier paper
\cite{BLO6}.

In the proof of Lemma A.8 in \cite{BLO6}, we applied \cite[Proposition
5.4]{BLO3} to orbit categories of transporter systems, while that
proposition is stated only for orbit categories of fusion systems. After
replacing that by Proposition \ref{p:lim*=Lambda*} below, Lemma A.8 can be
proven as stated in \cite{BLO6}.

In the proof of Corollary A.10 in \cite{BLO6}, we applied \cite[Proposition
A.9(b)]{BLO6} in a way that is not valid unless all objects in $\call$ are
$\calf$-centric. That corollary is a special case of Proposition
\ref{p:dropmore} here.

\section{Inclusions of linking systems for the same fusion system} 
\label{s:L0<L}

\newcommand{\lincl}{\iota}

We refer to \cite[Definition 1.9]{BLO6} for the complete definition of a 
linking system associated to a fusion system $\calf$ over a discrete 
$p$-toral group $S$. Very briefly, it consists of a category $\call$ whose 
objects are subgroups of $S$, together with a pair of functors 
	\[ \calt_{\Ob(\call)}(S) \Right4{\delta} \call \Right4{\pi} \calf 
	\]
satisfying certain conditions. Here, $\calt_{\Ob(\call)}(S)$ is the 
\emph{transporter category}: the category with the same objects as $\call$, 
and where $\Mor_{\calt_{\Ob(\call)}(S)}(P,Q)$ is the set of all $g\in S$ 
such that $\9gP\le Q$. Also, the set $\Ob(\call)$ is required to be 
closed under $\calf$-conjugacy and overgroups, and must include all 
subgroups of $S$ that are $\calf$-centric and $\calf$-radical. 

As usual, when $P\le Q$ are objects in a linking system $\call$, we write 
$\lincl_{P,Q}=\delta_{P,Q}(1)\in\Mor_\call(P,Q)$ for the inclusion morphism. 
By ``extensions'' and ``restrictions'' of morphisms we mean extensions and 
restrictions of source and target both, with respect to these inclusions. 


\begin{Prop} \label{p:L-prop}
The following hold for each linking system $\call$ associated to a saturated 
fusion system $\calf$ over a discrete $p$-toral group $S$.
\begin{enuma} 

\item For each $P,Q\in\Ob(\call)$, the homomorphism 
$\pi_{P,Q}\:\Mor_\call(P,Q)\too\homf(P,Q)$ is surjective. The group 
$\Ker(\pi_P)$ acts freely on $\Mor_\call(P,Q)$, and $\pi_{P,Q}$ induces a 
bijection 
	\[ \Mor_\call(P,Q)/\Ker(\pi_P) \Right5{\cong} \homf(P,Q) \,. \]
If $P$ is fully centralized in $\calf$, then 
$\Ker(\pi_P)=\delta_P(C_S(P))$. 

\item For every morphism $\psi\in\Mor_\call(P,Q)$, and every 
$P_*,Q_*\in\Ob(\call)$ such that $P_*\le{}P$, $Q_*\le{}Q$, and 
$\pi(\psi)(P_*)\le{}Q_*$, there is a unique morphism 
$\psi|_{P_*,Q_*}\in\Mor_\call(P_*,Q_*)$ (the ``restriction'' of $\psi$) 
such that $\psi\circ\iota_{P_*,P}=\iota_{Q_*,Q}\circ\psi|_{P_*,Q_*}$.  

\item Let $P\nsg \widebar{P}\le S$ and 
$Q\le\widebar{Q}\le S$ be objects in $\call$.  Let 
$\psi\in\Mor_\call(P,Q)$ be such that for each $g\in\widebar{P}$, 
there is $h\in\widebar{Q}$ satisfying 
$\iota_{Q,\widebar{Q}}\circ\psi\circ\delta_P(g)
=\delta_{Q,\widebar{Q}}(h)\circ\psi$.   Then there is a unique morphism 
$\widebar{\psi}\in\Mor_\call(\widebar{P},\widebar{Q})$ such that 
$\widebar{\psi}|_{P,Q}=\psi$. 

\item All morphisms in $\call$ are monomorphisms and 
epimorphisms in the categorical sense.

\item A morphism $\psi$ in $\call$ is an isomorphism in $\call$ if 
$\pi(\psi)$ is an isomorphism in $\calf$. 

\end{enuma}
\end{Prop}

\begin{proof} Points (a)--(d) are shown in \cite[Proposition 
A.4]{BLO6}, while (e) follows from Propositions A.2(c) and A.5 in 
\cite{BLO6}. 
\end{proof}

We want to show that the geometric realizations of two linking systems 
associated to the same fusion system are homotopy equivalent. The next 
lemma is a first step towards doing that.

\begin{Lem} \label{l:|L0|simeq|L|}
Let $\calf$ be a saturated fusion system over a discrete $p$-toral group 
$S$. Let $\call_0\subseteq\call$ be linking systems associated to $\calf$ 
such that $\Ob(\call)\sminus\Ob(\call_0)=\calp$, where $\calp$ is an 
$\calf$-conjugacy class of subgroups of $S$. Then the inclusion of nerves 
$|\call_0|\subseteq|\call|$ is a homotopy equivalence.
\end{Lem}

\newcommand{\rr}{\textbf{\textit{r}}}
\newcommand{\ii}{\textbf{\textit{j}}}
\renewcommand{\8}{{}^\bullet}

\begin{proof} The following proof is essentially that given in 
\cite[Proposition 3.11]{BCGLO1}, modified for fusion systems over discrete 
$p$-toral groups. To simplify notation, for 
$\varphi\in\Mor_{\call}(Q,R)$, we write 
$\Im(\varphi)=\Im(\pi(\varphi))\le{}R$, and 
$\varphi(Q_0)=\pi(\varphi)(Q_0)\le{}R$ if $Q_0\le{}Q$.

We must show that the inclusion functor $\cali\:\call_0\rightarrow \call$ 
induces a homotopy equivalence $|\call_0|\simeq |\call|$. By Quillen's 
Theorem A (see \cite{Quillen}), it will be enough to prove that the 
undercategory $P{\downarrow}\cali$ is contractible (i.e., $|P{\downarrow} 
\cali|\simeq *$) for each $P$ in $\call$. This is clear when 
$P\notin\calp$ (since $P{\downarrow}\cali$ has initial object $(P,\Id)$ in 
that case), so it suffices to consider the case $P\in\calp$. Since all 
subgroups in $\calp$ are isomorphic in the category $\call$, we can 
assume that $P$ is fully normalized.

Set 
	\[ \widehat{P}= \{x\in N_S(P) \,|\, c_x\in O_p(\autf(P)) \}. \] 
Recall that $P\notin\Ob(\call_0)$, and hence by definition of a linking 
system cannot be both $\calf$-centric and $\calf$-radical. If $P$ is not 
$\calf$-centric, then $\5P\ge P{\cdot}C_S(P)>P$. If $P$ is not 
$\calf$-radical, then $\Inn(P)<O_p(\autf(P))\le\Aut_S(P)$, the last 
inclusion since $P$ is fully normalized, and so $\5P>P$. Thus 
$\5P\in\Ob(\call_0)$ (and $P\nsg\5P$) in either case. 


Set $N_{\call_0}(P)=\call_0\cap N_{\call}(P)$, and let 
	\[ \cali_N\: N_{\call_0}(P) \Right5{} N_{\call}(P) \]
be the inclusion (thus the restriction of $\cali$). Consider the functors 
	\[ \5P{\downarrow}\cali_N \Right4{\ii_2} P{\downarrow}\cali_N 
	\Right4{\ii_1} P{\downarrow}\cali\,, \]
where $\ii_1$ is the inclusion of undercategories, induced by the 
inclusions $N_{\call_0}(P)\to \call_0$ and $N_\call(P)\to\call$, and where 
$\ii_2$ sends an object $(Q,\alpha)$ to $(Q,\alpha\circ\iota_{P,\5P})$ 
(for $\alpha\in\Mor_{N_{\call}(P)}(\5P,Q)$). 
For each $i=1,2$, we will construct a retraction $\rr_i$ such that 
$\rr_i\circ\ii_i=\Id$, together with a natural transformation 
of functors between $\ii_i\circ\rr_i$ and the identity. It then follows 
that $|P{\downarrow}\cali|\simeq |P{\downarrow}\cali_N|\simeq 
|\5P{\downarrow}\cali_N|$, and the last space is contractible since 
$(\5P,\Id_{\5P})$ is an initial object in $\5P{\downarrow}\cali_N$ (recall 
$\5P\in\Ob(N_{\call_0}(P))$). 
 
\smallskip

\noindent\textbf{Step 1: } We first construct the 
retraction functor $\rr_1\:P{\downarrow}\cali\too P{\downarrow}\cali_N$, 
together with a natural transformation $\bigl(\ii_1\circ 
\rr_1\xto{~\eta~}\Id_{P{\downarrow}\cali}\bigr)$.  By \cite[Lemma 
1.7(c)]{BLO6}, for each $R\in\calp$ (the $\calf$-conjugacy class of $P$), 
there is a morphism in $\calf$ from $N_S(R)$ to $N_S(P)$ which sends $R$ 
isomorphically to $P$. Since $\pi$ is surjective on morphism sets by 
Proposition \ref{p:L-prop}(a), we can choose a morphism 
	\[ \Phi_{R}\in\Mor_{\call}(N_S(R),N_S(P)), \]
for each $R\in\calp$ such that $\Phi_{R}(R)=P$. 
We also require that $\Phi_P=\Id_{N_S(P)}$.

For each $(Q,\varphi)$ in $P{\downarrow}\cali$, set 
$N_{(Q,\varphi)}(P)=\Phi_{\varphi(P)}(N_Q(\varphi(P)))$. Thus 
$P<N_{(Q,\varphi)}(P)\le N_S(P)$, since $N_Q(\varphi(P))>\varphi(P)$ by 
\cite[Lemma 1.8]{BLO3}. Consider the following diagram:
	\beqq \vcenter{\xymatrix@C=50pt@R=30pt{ 
	P \ar[r]^-{\varphi_*} \ar[dr]_-{\varphi'} & N_Q(\varphi(P)) 
	\ar[r]^-{\iota} \ar[d]^(0.4){(\Phi_{\varphi(P)})_*}_(0.4){\cong} & Q \\
	& N_{(Q,\varphi)}(P) \ar[ur]_-{\varphi''}
	}} \label{e:LO-0} \eeqq
where $\iota$ is the inclusion (in $\call$), where $\varphi_*$ and 
$(\Phi_{\varphi(P)})_*$ denote the restriction morphisms of Proposition 
\ref{p:L-prop}(b) (restricting source and/or target as indicated), and where 
$\varphi'=(\Phi_{\varphi(P)})_*\circ\varphi_*$ and 
$\varphi''=\iota\circ(\Phi_{\varphi(P)})_*^{-1}$ (so the triangles 
commute). Define $\rr_1:P{\downarrow}\cali \too P{\downarrow}\cali_N$ on 
objects by setting
	\[ \rr_1\bigl(Q,\varphi\bigl) = 
	\bigl(N_{(Q,\varphi)}(P), \varphi'\bigr). \]

We still need to define $\rr_1$ on morphisms. 
For each morphism $\beta\in\Mor_{P{\downarrow}\cali} 
((Q,\varphi),(R,\psi))$; i.e., for each commutative square of the form
	\beqq \vcenter{\xymatrix@C=40pt@R=25pt{ 
	P \ar[r]^{\varphi} \ar[d]_{\Id} & Q \ar[d]_{\beta} \\
	P \ar[r]^{\psi} & R \rlap{\,,}
	}} \label{e:L0-1} \eeqq
we get the following commutative diagrams:
	\beqq  \vcenter{\xymatrix@C=40pt@R=30pt{ 
	P \ar[r]^-{\varphi_*} \ar[d]_{\Id} & N_Q(\varphi(P)) \ar[r]^-{\iota} 
	\ar[d]_{\beta_*} & Q \ar[d]_{\beta} \\ 
	P \ar[r]^-{\psi_*} & N_R(\psi(P)) \ar[r]^-{\iota} & R 
	}} \qquad\qquad
	\vcenter{\xymatrix@C=40pt@R=30pt{ 
	N_Q(\varphi(P)) \ar[r]^{(\Phi_{\varphi(P)})_*}_{\cong} \ar[d]_{\beta_*} 
	& N_{(Q,\varphi)}(P) \ar[d]_{\5\beta} \\
	N_R(\psi(P)) \ar[r]^{(\Phi_{\psi(P)})_*}_{\cong} & N_{(R,\psi)}(P) 
	}} \label{e:LO-2} \eeqq
Again, $\beta_*$, $\varphi_*$, and $(\Phi_{\varphi(P)})_*$ denote the 
morphisms after restricting source and target as shown, while 
$\5\beta=(\Phi_{\psi(P)})_*\circ\beta_*\circ(\Phi_{\varphi(P)})_*^{-1}$.

Since each of the three squares in \eqref{e:LO-2} commutes, it follows that 
$\psi'=\5\beta\circ\varphi'$ and $\psi''\circ\5\beta=\varphi''$, where 
$\varphi'$, $\psi'$, $\varphi''$, and $\psi''$ are as in \eqref{e:LO-0}. So 
we can define $\rr_1$ on morphisms by setting 
	\[ \rr_1(\beta) = \5\beta\colon (N_{(Q,\varphi)}(P),\varphi') 
	\Right3{} (N_{(R,\psi)}(P),\psi'). \] 

Define a natural transformation $\eta\colon 
\ii_1\circ\rr_1\to\Id_{P{\downarrow}\cali}$ by sending an object $(Q,\varphi)$ 
to the morphism 
	\[\varphi''\: (\ii_1\circ\rr_1)(Q,\varphi) = 
	(N_{(Q,\varphi)}(P),\varphi') \Right3{} (Q,\varphi) \] 
(a natural transformation since $\psi''\circ\5\beta=\varphi''$). 
Since $\rr_1\circ\ii_1=\Id_{P{\downarrow}\cali_N}$, this finishes the proof 
that $|P{\downarrow} \cali|\simeq |P{\downarrow} \cali_N|$. 

\smallskip

\noindent\textbf{Step 2: } 
Recall that $\ii_2:\widehat{P}{\downarrow} \cali_N \rightarrow 
P{\downarrow}\cali_N$ is induced by precomposing with 
the inclusion $\lincl_{P,\widehat{P}}\in\Mor_{\call}(P,\widehat{P})$. 
We now construct a retraction functor 
$\rr_2\:P{\downarrow}\cali_N\too\5P{\downarrow}\cali_N$, together with a 
natural transformation of functors from $\Id_{P{\downarrow}\cali_N}$ to 
$\ii_2\circ\rr_2$. 

Since $P$ is fully normalized in $\calf$ by assumption, it is also fully 
centralized by the Sylow axiom for a saturated fusion system 
\cite[Definition 1.4(I)]{BLO6}. So by Proposition \ref{p:L-prop}(a), the 
structure homomorphism $\pi_P\:\Aut_{\call}(P)\too\autf(P)$ is 
surjective and $\Ker(\pi_P)=\delta_P(C_S(P))$. Since 
$\pi_P\circ\delta_P\:N_S(P)\too\autf(P)$ sends $g$ to $c_g$, 
we have $\5P=\delta_P^{-1}(\pi_P^{-1}(O_p(\autf(P))))$, and so 
$\delta_P(\5P)=\pi_P^{-1}(O_p(\autf(P)))$. Since $\pi_P$ is surjective, 
this proves that $\delta_P(\5P)\nsg\Aut_{\call}(P)$.

Fix subgroups $Q,R\le N_S(P)$ containing $P$, and let 
$\varphi\in\Mor_{N_{\call}(P)}(Q,R)$ be a morphism. For $g\in Q\5P$, 
write $g=g'x$ for $g'\in Q$ and $x\in\5P$, and let $y\in\5P$ be such that 
$\delta_P(y)=(\varphi|_P)\delta_P(x)(\varphi|_P)^{-1}$. Then 
$\varphi\delta_Q(g')=\delta_R(\varphi(g'))\varphi$ by condition (C) in the 
definition of a linking system (\cite[Definition 1.9]{BLO6}), so after 
restriction of source and target to $P$, we get 
$(\varphi|_P)\delta_P(g')(\varphi|_P)^{-1}=\delta_P(\varphi(g'))$. 
Set $h=\varphi(g')y$. Then 
	\[ (\varphi|_P)\delta_P(g)(\varphi|_P)^{-1} = 
	(\varphi|_P)\delta_P(g')(\varphi|_P)^{-1}\circ
	(\varphi|_P)\delta_P(x)(\varphi|_P)^{-1} = 
	\delta_P(\varphi(g')y) = \delta_P(h) \]
in $\Aut_{\call}(P)$, and hence 
	\begin{align*} 
	\iota_{R,R\5P} \circ \varphi \circ \delta_Q(g) \circ \iota_{P,Q} 
	&= \iota_{P,R\5P}\circ(\varphi|_P)\circ\delta_P(g) \\
	&= \iota_{P,R\5P}\circ \delta_P(h)\circ(\varphi|_P) = 
	\delta_{R,R\5P}(h) \circ \varphi \circ \iota_{P,Q} \in 
	\Mor_{N_{\call_0}(P)}(P,R\5P). 
	\end{align*}
So $\iota_{R,R\5P}\circ\varphi\circ\delta_Q(g)=\delta_{R,R\5P}(h)\circ 
\varphi$ since $\iota_{P,Q}$ is an epimorphism by Proposition 
\ref{p:L-prop}(d). Hence by Proposition \ref{p:L-prop}(c), there is a 
unique morphism $\5\varphi\in\Mor_{\call_0}(Q\5P,R\5P)$ (the ``extension'' 
of $\varphi$) such that the following diagram commutes in $\call$: 
	\beq \vcenter{\xymatrix@C=40pt@R=25pt{ 
	Q \ar[r]^{\varphi} \ar[d]^{\lincl_{Q,Q\5P}}
	& R \ar[d]_{\lincl_{R,R\5P}} \\
	Q\5P \ar[r]^{\5\varphi} & R\5P 
	}} \eeq
Note that $Q\5P,R\5P\in\Ob(\call_0)$ (and hence $\5\varphi$ lies in 
$N_{\call_0}(P)$) since $\5P>P$.

The functor $\rr_2\colon P\downarrow \cali_N\to 
\widehat{P}{\downarrow}\cali_N$ is defined on objects by setting
	\[ \rr_2\bigl(Q, \alpha\bigr)=
	\bigl(Q\widehat{P}, \widehat{\alpha}\bigr). \] 
If $\beta\colon (Q,\alpha)\too (R, \gamma)$ is a morphism in 
$P{\downarrow}\cali_N$, that is
$\beta\in\Mor_{\call}(Q,R)$ is such that 
$\beta\circ\alpha=\gamma$, then we define $\rr_2(\beta)=\widehat{\beta}$. 
Because of the uniqueness of the extension $\widehat{\beta}$, this 
construction defines a functor from $P{\downarrow}\cali_N$ to 
$\5P{\downarrow}\cali_N$. Moreover, $\rr_2\circ 
\ii_2=\Id_{\widehat{P}{\downarrow} \cali_N}$, and $\ii_2\circ \rr_2\simeq 
\Id_{P{\downarrow}\cali_N}$, where the homotopy is induced by the natural 
transformation given by the inclusions $\lincl_{Q,Q\widehat{P}}$.
\end{proof}

Lemma \ref{l:|L0|simeq|L|} says that under certain conditions, we can 
remove one class of objects from a linking system without changing the 
homotopy type of its nerve. We need to combine this with the ``bullet 
construction'', defined in \cite{BLO3}, to show that we can remove 
infinitely many classes of objects without changing the homotopy type. 

\begin{Defi}[{\cite[Definition 3.1]{BLO3}}] \label{d:bullet}
Let $\calf$ be a fusion system over 
a discrete $p$-toral group $S$, and let $T\nsg S$ be its identity 
component. Let $m\ge0$ be such that $S/T$ has exponent $p^m$; thus 
$g^{p^m}\in T$ for all $g\in S$. Set $W=\autf(T)$.
\begin{enuma} 

\item For each $A\le T$, set $I(A)=C_T(C_W(A))$, and let $I(A)_0$ be its 
identity component. 

\item For each $P\le S$, set $P^{[m]}=\gen{g^{p^m}\,|\,g\in P} \le T$, and 
set $P\8=P\cdot I(P^{[m]})_0$. 

\item Let $\calf\8\subseteq\calf$ be the full subcategory with 
$\Ob(\calf\8)=\{P\8\,|\,P\le S\}$. 

\end{enuma}
\end{Defi}

Some of the key properties of the bullet construction are listed in the 
following proposition:

\begin{Prop}[{\cite{BLO3}}] \label{p:bullet}
Let $\calf$ be a saturated fusion system over a discrete $p$-toral group 
$S$. 
\begin{enuma} 

\item The set $\Ob(\calf\8)=\{P\8\,|\,P\le S\}$ contains only finitely many 
$S$-conjugacy classes of subgroups.

\item For each $P\le S$, we have $(P\8)\8=P\8$. 

\item For each $P\le Q\le S$, we have $P\8\le Q\8$. 


\item For each $P,Q\le S$ and each $\varphi\in\homf(P,Q)$, there is a 
unique homomorphism $\varphi\8\in\homf(P\8,Q\8)$ such that 
$\varphi\8|_P=\varphi$. 

\end{enuma}
\end{Prop}

\begin{proof} Points (a)--(c) are stated as Lemma 3.2 in \cite{BLO3}, and 
(d) is stated as Proposition 3.3. 
\end{proof}

In particular, Proposition \ref{p:bullet}(b,c,d) implies that there is a well 
defined retraction functor $\calf\too\calf\8$ that sends an object $P$ to 
$P\8$ and a morphism $\varphi$ to $\varphi\8$. 

The following lemma is shown in \cite[Proposition 4.5]{BLO3} when 
$\call$ is a centric linking system, but we need it in a more general 
situation. 

\begin{Lem} \label{l:psi.}
Let $\calf$ be a saturated fusion system over a discrete $p$-toral group 
$S$, and let $\call$ be a linking system associated to $\calf$. 
\begin{enuma} 

\item For each $P,Q\in\Ob(\call)$ and each $\psi\in\Mor_\call(P,Q)$, there 
is a unique morphism $\psi\8\in\Mor_\call(P\8,Q\8)$ that restricts to 
$\psi$, and is such that $\pi(\psi\8)=(\pi(\psi))\8$. 

\item The space $|\call\8|$ is a deformation retract of $|\call|$, 
with retraction $|\call|\too|\call\8|$ induced by the functor that sends 
$P$ to $P\8$ and $\psi$ to $\psi\8$.

\end{enuma}
\end{Lem}

\begin{proof} \textbf{(a) } It suffices to show this when $P$ is fully 
normalized in $\calf$. Set $\varphi=\pi(\psi)\in\homf(P,Q)$. By Proposition 
\ref{p:bullet}(d), there is a unique $\varphi\8\in\homf(P\8,Q\8)$ that 
extends $\varphi$. By \cite[Proposition 1.13(v)]{Gonzalez}, we have 
$C_S(P)=C_S(P\8)$, and by Proposition \ref{p:L-prop}(a), this group acts 
freely and transitively on the sets $\pi_{P,Q}^{-1}(\varphi)$ and 
$\pi_{P\8,Q\8}^{-1}(\varphi\8)$. The restriction map from 
$\pi_{P\8,Q\8}^{-1}(\varphi\8)$ to $\pi_{P,Q}^{-1}(\varphi)$ commutes with 
this action, and hence is a bijection.

\smallskip

\noindent\textbf{(b) } Point (a) implies that there is a well defined functor 
$(-)\8\:\call\too\call\8$ that sends $P$ to $P\8$ and $\psi$ to $\psi\8$. 
This defines a map $r\:|\call|\too|\call\8|$, which is a retraction by 
Proposition \ref{p:bullet}(b). Let $i$ denote the inclusion; then $i\circ 
r$ homotopic to the identity on $|\call|$ since there is a natural 
transformation of functors from $(-)\8$ to the identity that sends each 
object $P$ to the inclusion $\iota_{P,P\8}$ from $P$ to $P\8$. 
\end{proof}

We now combine Lemmas \ref{l:|L0|simeq|L|} and \ref{l:psi.} to get the 
result we need.

\begin{Prop} \label{p:dropmore}
Let $\calf$ be a saturated fusion system over a discrete $p$-toral group 
$S$, and let $\call_0\subseteq\call$ be a pair of linking systems 
associated to $\calf$. Then the inclusion of $|\call_0|$ in $|\call|$ is a 
homotopy equivalence.
\end{Prop}

\newcommand{\xxset}{\textbf{\textit{L}}}

\begin{proof} Let $\xxset$ be the set of all linking subsystems 
$\call'\subseteq\call$ containing $\call_0$ such that the inclusion 
$|\call_0|\subseteq|\call'|$ is a homotopy equivalence. We must show that 
$\call\in\xxset$.

Assume otherwise, and choose $\call_1\in\xxset$ for which 
$\Ob(\call_1\8)$ contains the largest possible number of 
$\calf$-conjugacy classes.  Let $\call_2\subseteq\call$ be the full 
subcategory with $\Ob(\call_2)=\{P\in\Ob(\call)\,|\,P\8\in\Ob(\call_1)\}$.  
By Lemma \ref{l:psi.}(b), $|\call_1\8|=|\call_2\8|$ is a 
deformation retract of $|\call_1|$ and of $|\call_2|$, so 
$|\call_1|\simeq|\call_2|$, and $\call_2\in\xxset$. 

Since $\call_2\subsetneqq\call$ by the assumption that 
$\call\notin\xxset$, we have $\Ob(\call\8)\not\subseteq\Ob(\call_2)$. 
Let $P$ be maximal among objects in $\call\8$ not in $\call_2$.  By 
definition of $\Ob(\call_2)$, $P$ is maximal among all objects of $\call$ 
not in $\call_2$.  Let $\call_3\subseteq\call$ be the full subcategory with 
$\Ob(\call_3)=\Ob(\call_2)\cup{}P^\calf$.  

By Lemma \ref{l:|L0|simeq|L|}, the inclusion of $|\call_2|$ into 
$|\call_3|$ is a homotopy equivalence.  Hence $\call_3\in\xxset$, 
contradicting our maximality assumption on $\call_1\8$.  
\end{proof}


\section{\texorpdfstring{$\Lambda$}{Lambda}-functors}
\label{s:Lambda}

For any group $G$, let $\calo_p(G)$ be the \emph{$p$-subgroup orbit 
category} of $G$: the category whose objects are the $p$-subgroups of $G$, 
and where 
	\[ \Mor_{\calo_p(G)}(P,Q) =  \bigl\{Qg\in Q{\backslash}G \,\big|\, 
	\9gP\le Q \bigr\} = Q{\backslash}\{g\in G\,|\, \9gP\le Q\}. \]
Note that there is a bijection 
	\[ \Mor_{\calo_p(G)}(P,Q) \Right3{\cong} \map_G(G/P,G/Q) \]
that sends a coset $Qg$ to the $G$-map $\bigl(xP\mapsto xg^{-1}Q\bigr)$.

\begin{Defi} \label{d:Lambda}
Let $G$ be a locally finite group, and let $M$ be a $\Z G$-module. 
Define a functor $F_M\:\calo_p(G)\op\too\Ab$ by setting
	\[ F_M(P) = \begin{cases} 
	M & \textup{if $P=1$} \\
	0 & \textup{if $P\ne1$}
	\end{cases} \]
for each $p$-subgroup $P\le G$. Here, $\Aut_{\calo_p(G)}(1)\cong G$ has 
the given action on $M$. Set 
	\[ \Lambda^*(G;M) = \higherlim{\calo_p(G)}*(F_M) . \]
\end{Defi}

We refer to \cite[Definition A.1]{BLO6} for the definition of a 
transporter system. As in Section 1, when $S$ is a group and $\calh$ is a set 
of subgroups of $S$, we let $\calt_\calh(S)$ be the transporter category of $S$ 
with object set $\calh$, where $\Mor_{\calt_\calh(S)}(P,Q)$ is the 
set of all $g\in S$ such that $\9gP\le Q$. A transporter 
system associated to a fusion system $\calf$ over a discrete $p$-toral 
group $S$ consists of a category $\calt$ whose objects are subgroups of 
$S$, together with a pair of functors 
	\[ \calt_{\Ob(\calt)}(S) \Right4{\gee} \calt \Right4{\rho} \calf, 
	\]
such that $\gee$ is the identity on objects and injective on morphism sets, 
$\rho$ is the inclusion on objects and surjective on morphism sets, and 
several other conditions are satisfied. The only requirements on the set 
$\Ob(\calt)$ are that it be nonempty and closed under $\calf$-conjugacy and 
overgroups. 

As with linking systems, when $P\le Q$ are objects in a transporter system 
$\calt$, we write $\iota_{P,Q}=\gee_{P,Q}(1)\in\Mor_\calt(P,Q)$ for the 
inclusion morphism. By ``extensions'' and ``restrictions'' of morphisms we 
mean extensions and restrictions of source and target both, with respect to 
these inclusions. We will only need to refer to the following properties:

\begin{Prop} \label{p:T-prop}
The following hold for any transporter system $\calt$ associated to 
a saturated fusion system $\calf$ over a discrete $p$-toral group $S$.
\begin{enuma} 

\item For each $P,Q\in\Ob(\calt)$, the composite 
$\rho_{P,Q}\circ\gee_{P,Q}$ sends $g\in N_S(P,Q)$ to the homomorphism 
$c_g=(x\mapsto {}^gx) \in\homf(P,Q)$. 

\item Let $P\nsg \widebar{P}\le S$ and $Q\le\widebar{Q}\le S$ be objects in 
$\calt$.  Let $\psi\in\Iso_\calt(P,Q)$ be such that for each 
$g\in\widebar{P}$, there is $h\in\widebar{Q}$ satisfying 
$\iota_{Q,\widebar{Q}}\circ\psi\circ\gee_P(g) 
=\gee_{Q,\widebar{Q}}(h)\circ\psi$.   Then $\psi$ extends to a unique 
morphism $\widebar{\psi}\in\Mor_\calt(\widebar{P},\widebar{Q})$ such that 
$\widebar{\psi}|_{P,Q}=\psi$. 

\item All morphisms in $\calt$ are monomorphisms and epimorphisms in the 
categorical sense.

\end{enuma}
\end{Prop}

\begin{proof} Point (a) is axiom (B) in \cite[Definition A.1]{BLO6}, 
and point (c) is shown in \cite[Proposition A.2(d)]{BLO6}.

The existence of an extension in 
point (b) is axiom (II) in \cite[Definition A.1]{BLO6}, except that 
it is stated there under the additional assumption that $Q$ is normal in 
$\widebar{Q}$. But if $Q\le\widebar{Q}$ is not normal, and 
$\psi\in\Iso_\calt(P,Q)$, $g\in\widebar{P}$, and $h\in\widebar{Q}$ are as 
above, then 
	\[ Q = \Im(\rho(\iota_{Q,\4Q}\circ\psi\circ\gee_P(g)) ) = 
	\Im(\rho(\gee_{Q,\widebar{Q}}(h)\circ\psi))
	= {}^hQ, \]
the last equality by (a), and so $h\in N_{\widebar{Q}}(Q)$. Hence $\psi$ 
extends to $\5\psi\in\Mor_\calt(\widebar{P},N_{\widebar{Q}}(Q))$ by axiom 
(II) as stated in \cite{BLO6}, and we can take 
$\widebar{\psi}=\iota_{N_{\4Q}(Q),\4Q}\circ\5\psi$. 

The extension in (b) is unique since inclusion morphisms in $\calt$ are 
epimorphisms by (c). 
\end{proof}

If $\calt$ is a transporter system over a discrete $p$-toral group $S$, 
then its orbit category $\calo(\calt)$ is the category with the same 
objects, and where for each $P,Q\in\Ob(\calt)$,
	\[ \Mor_{\calo(\calt)}(P,Q) = \gee_Q(Q){\backslash}\Mor_\calt(P,Q). 
	\]
Thus, for example, if $\calt$ is the transporter system of a group $G$ (with 
some set of objects), then $\calo(\calt)$ is the orbit category of $G$ in 
the above sense.

\begin{Prop} \label{p:lim*=Lambda*}
Let $\calt$ be a transporter system associated to a saturated fusion system 
$\calf$ over a discrete $p$-toral group $S$. Fix $P\in\Ob(\calt)$, and let 
	\[ \Phi\: \calo(\calt)\op \Right4{} \Ab \]
be a functor such that $\Phi(Q)=0$ for each $Q\notin P^\calf$. Then 
	\[ \higherlim{\calo(\calt)}*(\Phi) \cong 
	\Lambda^*(\Aut_{\calo(\calt)}(P);\Phi(P)). \]
\end{Prop}

\begin{proof} By axiom (I) for a transporter system (see 
\cite[Definition A.1]{BLO6}), there is $Q\in P^\calf$ such that the index 
of $\gee_Q(N_S(Q))$ in $\Aut_\calt(Q)$ is finite and prime to $p$. Thus 
$\Aut_\calt(Q)$ is an extension of a discrete $p$-toral group by a finite 
group, and we write $\gee_Q(N_S(Q))\in\sylp{\Aut_\calt(Q)}$ for short. 
Since $\Lambda^*(\Aut_{\calo(\calt)}(Q);\Phi(Q))\cong 
\Lambda^*(\Aut_{\calo(\calt)}(P);\Phi(P))$ for each $Q\in P^\calf$, it 
suffices to prove the proposition when 
$\gee_P(N_S(P))\in\sylp{\Aut_\calt(P)}$. We will show that this 
is a special case of \cite[Proposition 5.3]{BLO3}. 

Set $\Gamma=\Aut_{\calo(\calt)}(P)=\gee_P(P){\backslash}\Aut_\calt(P)$. To 
simplify the notation, we identify $N_S(P)$ with its image 
$\gee_P(N_S(P))\le\Aut_\calt(P)$, and identify $N_S(P)/P$ with 
$\gee_P(N_S(P))/\gee_P(P)$. 

Let $\calh$ be the set of all subgroups of $N_S(P)/P\in\sylp\Gamma$. Define 
	\[ \alpha \: \calo_\calh(\Gamma) \Right4{} \calo(\calt) \]
by setting $\alpha(Q/P)=Q$, and 
	\[ \alpha(Q/P \xto{~R\gamma~} R/P) = [\4\gamma]\in 
	\Mor_{\calo(\calt)}(Q,R) \]
where $\gamma\in\Aut_\calt(P)$ extends to a unique morphism 
$\4\gamma\in\Mor_\calt(Q,R)$ by Proposition \ref{p:T-prop}(b). 

The proposition will follow immediately from \cite[Proposition 5.3]{BLO3} 
once we have shown that conditions (a)--(d) in the proposition all hold. 
Set $c_0=\alpha(1)$, following the notation used in that proposition, and 
note that $c_0=P$. 
\begin{enuma} 

\item For each $P\in\Ob(\calt)$ and each $\psi\in\End_\calt(P)$, the 
homomorphism $\rho(\psi)\in\End_\calf(P)$ is an isomorphism of groups, and 
hence $\psi\in\Aut_\calt(P)$ by \cite[Proposition A.2(c)]{BLO6}. So 
by construction, $\alpha$ sends 
$\Gamma=\Aut_{\calo_\calh(\Gamma)}(1)$ isomorphically to 
$\End_{\calo(\calt)}(P)=\Aut_{\calo(\calt)}(P)$. 

\item Let $U\in\Ob(\calt)$ be such that $U\notin P^\calf$. We must show 
that all isotropy subgroups of the $\Gamma$-action on 
$\Mor_{\calo(\calt)}(P,U)$ are nontrivial and conjugate in $\Gamma$ to 
subgroups in $\calh$. Since $\calh$ is the set of all subgroups of a Sylow 
$p$-subgroup of $\Gamma$, where $\Gamma$ has a discrete $p$-toral 
subgroup of finite index, this means showing that each isotropy subgroup 
is a nontrivial discrete $p$-toral subgroup. 

Fix $\psi\in\Mor_\calt(P,U)$, and let $[\psi]$ be its class in 
$\Mor(\calo(\calt))$. Set $Q=\rho(\psi)(P)<U$. Then 
$\psi=\iota_{Q,U}\circ(\psi|_{P,Q})$, where $\psi|_{P,Q}\in\Iso_\calt(P,Q)$. 
So the isotropy subgroup for the action of $\Gamma$ on $\psi$ is isomorphic 
to the isotropy subgroup of $\Aut_{\calo(\calt)}(Q)$ on $\iota_{Q,U}$. 

For each $\gamma\in\Aut_\calt(Q)$, we have 
$[\iota_{Q,U}\gamma]=[\iota_{P,U}]$ 
if and only if there is $g\in U$ such that ${}^gQ=Q$ and 
$\gamma=\gee_U(g)|_{Q,Q}=\gee_Q(g)$. Thus the isotropy subgroup is the 
group of all $[\gee_Q(g)]$ for $g\in N_U(Q)$, hence isomorphic to 
$N_U(Q)/Q$, which is a nontrivial discrete $p$-toral group by 
\cite[Lemma 1.8]{BLO3} and since $Q<U$ are both discrete $p$-toral groups.

\item We claim that all morphisms in $\calo(\calt)$ are epimorphisms in the 
categorical sense. To see this, fix subgroups $P,Q,R\in\Ob(\calt)$ and 
morphisms $\varphi\in\Mor_{\calt}(P,Q)$ and 
$\alpha,\beta\in\Mor_{\calt}(Q,R)$ such that 
$[\alpha][\varphi]=[\beta][\varphi]$, where $[-]$ denotes the class in 
$\calo(\calt)$ of a morphism in $\calt$. Thus there is $g\in R$ such that 
$\alpha\varphi=\gee_R(g)\beta\varphi$ in $\calt$. 

By Proposition \ref{p:T-prop}(c), all morphisms in $\calt$ are 
epimorphisms. So $\alpha=\gee_R(g)\beta$, and hence $[\alpha]=[\beta]$, 
proving that $\varphi$ is an epimorphism in $\calo(\calt)$. 


\item Fix $Q/P\le N_S(P)/P$, $U\in\Ob(\calt)$, and 
$\varphi\in\Mor_\calt(P,U)$ such that $[\varphi\gee_P(g)]=[\varphi]$ in 
$\Mor_{\calo(\calt)}(P,U)$ for each $g\in Q$. We must show that there is 
$\4\varphi\in\Mor_\calt(Q,U)$ such that $[\4\varphi|_{P,U}]=[\varphi]$ in 
$\Mor_{\calo(\calt)}(P,U)$. 

By assumption, for each $g\in Q$, there is $u\in U$ such that 
$\gee_U(u)\varphi=\varphi\gee_P(g)$. 
Set $U_0=\rho(\varphi)(P)\le U$. Thus $\varphi=\iota_{U_0,U}\varphi_0$, 
where $\varphi_0=\varphi|_{P,U_0}\in\Iso_\calt(P,U_0)$. So 
	\[ \gee_U(u) \iota_{U_0,U}\varphi_0= \iota_{U_0,U} 
	\varphi_0\gee_P(g), \]
hence $U_0=\Im(\rho(\varphi_0\gee_P(g))) 
=\Im(\rho(\gee_U(u)\iota_{U_0,U}\varphi_0)) ={}^uU_0$ (the last equality 
by Proposition \ref{p:T-prop}(a)), and so $u\in N_U(U_0)$. 
Thus $\iota_{U_0,U}\varphi_0\gee_P(g)\varphi_0^{-1} = 
\gee_U(u)\iota_{U_0,U}=\iota_{U_0,U}\gee_{U_0}(u)$. Since morphisms in a 
transporter category are monomorphisms by Proposition \ref{p:T-prop}(c), 
this implies that $\varphi_0\gee_P(g)\varphi_0^{-1} 
\in\gee_{U_0}(N_U(U_0))$. So by Proposition \ref{p:T-prop}(b), there is 
$\4\varphi\in\Mor_\calt(Q,U)$ such that $\4\varphi|_{P,U_0}=\varphi_0$, and 
hence $\4\varphi|_{P,U}=\varphi$. 
\qedhere
\end{enuma}
\end{proof}



\newpage

\end{document}